\documentclass[english]{article}
\usepackage[T1]{fontenc}
\usepackage[latin9]{inputenc}
\usepackage{mathrsfs}
\usepackage{amsmath}
\usepackage{amsthm}
\usepackage{amssymb}

\makeatletter
  \theoremstyle{plain}
  \newtheorem*{thm*}{\protect\theoremname}
\theoremstyle{plain}
\newtheorem{thm}{\protect\theoremname}
  \theoremstyle{definition}
  \newtheorem{defn}[thm]{\protect\definitionname}
  \theoremstyle{plain}
  \newtheorem{lem}[thm]{\protect\lemmaname}
  \theoremstyle{plain}
  \newtheorem{cor}[thm]{\protect\corollaryname}

\usepackage{fullpage}

\makeatother

\usepackage{babel}
  \providecommand{\corollaryname}{Corollary}
  \providecommand{\definitionname}{Definition}
  \providecommand{\lemmaname}{Lemma}
  \providecommand{\theoremname}{Theorem}
\providecommand{\theoremname}{Theorem}

\begin{document}

\title{Uniform Ergodicity for Brownian Motion in a Bounded Convex Set}

\author{Jackson Loper}
\maketitle
\begin{abstract}
We consider an $n$-dimensional Brownian Motion trapped inside a bounded
convex set by normally-reflecting boundaries. It is well-known that
this process is uniformly ergodic. However, the rates of this ergodicity
are not well-understood, especially in the regime of very high-dimensional
sets. Here we present new bounds on these rates for convex sets with
a given diameter. Our bounds do not depend upon the smoothness of
the boundary nor the value of the ambient dimension, $n$.
\end{abstract}
\global\long\def\tv#1{\left\Vert #1\right\Vert _{\mathrm{TV}}}
 \global\long\def\ltn#1{\left\Vert #1\right\Vert _{\mathscr{L}^{2}}}
 \global\long\def\pt#1#2#3{p\left(#1,#2,#3\right)}
 \global\long\def\FF#1#2{F_{d}\left(#1,#2\right)}

\section{Introduction}

Let $\Omega$ be an open, convex set with diameter $d$. We will consider
a Brownian motion, $\left\{ X_{t}\right\} _{t\geq0}$, trapped inside
$\Omega$ by normally-reflecting boundaries, initialized at some point
$X_{0}=x$. We will give a more precise definition of this process
momentarily; intuitively, $X$ is a homogeneous Markov process which
behaves like a Brownian motion on $\Omega$, spends essentially no
time at the boundary (which we will denote $\partial\Omega$), and
is always contained in the closure (which we will denote $\bar{\Omega}$).
Let $\left\{ \pt tx{dy}\right\} _{t\geq0,x\in\bar{\Omega}}$ denote
the transition measures of the process, i.e. $\mathbb{P}(X_{t}\in A|X_{0}=x)=\int_{A}\pt tx{dy}$.
It is well-known that the process $X$ is ergodic, with stationary
distribution $\sigma$, where $\sigma(dy)\triangleq dy/\mathrm{Vol}(\Omega)$.
Thus for every $x$ we must have that 
\[
\pt tx{\cdot}\rightarrow\sigma\quad\text{as}\quad t\rightarrow\infty
\]
in various modes of convergence. The rate and exact nature of this
convergence has been investigated in a number of ways. For example,
consider the case that $\Omega$ is a convex polytope which can be
contained inside a cube of diameter $d$, such as $[0,d/\sqrt{n}]^{n}$.
In this special case, it has long been known that 
\begin{equation}
\sup_{x}\tv{\pt tx{\cdot}-\sigma}\leq\sqrt{\frac{d^{2}}{2\pi t}}\label{eq:matthews}
\end{equation}
where $\tv{\mu}\triangleq\sup_{A}\left|\mu(A)\right|$ (cf. \cite{matthews1990mixing}).
Another important result follows from a certain Poincare constant
on arbitrary bounded convex domains, first rigorously shown by Bebendorf
in \cite{bebendorf2003note}. Using this result one may readily show
that 
\begin{gather}
\begin{array}{l}
\left|\int g(x)\left(\int f(y)\pt tx{dy}\right)dy-\int f(y)\sigma(dy)\right|\\
\leq\exp\left(-\frac{1}{2}\left(\frac{\pi}{d}\right)^{2}t\right)\ltn{g-\frac{1}{\mathrm{Vol}(\Omega)}}\ltn f
\end{array}\label{eq:bebendorf}
\end{gather}
where we define $\ltn f^{2}\triangleq\int f^{2}dx$ and take $g$
to be any density (i.e. $g\geq0$ and $\int g=1$). One way to see
the significance of this formula is to take $Y$ to be some variable
with $Y\sim\sigma$, and $\rho(dx)=g(x)dx$ to be some initial distribution.
Equation \ref{eq:bebendorf} can then be understood as a bound on
the rate at which $\mathbb{E}\left[f(X_{t})|X_{0}\sim\rho\right]\rightarrow\mathbb{E}\left[f(Y)\right]$.
Comparing Bebendorf's result with Equation \ref{eq:matthews}, we
see that Bebendorf's result has several advantages: 
\begin{enumerate}
\item Equation \ref{eq:matthews} becomes less powerful as $n\rightarrow\infty$,
but Equation \ref{eq:bebendorf} does not suffer from this deficit.
To see how Equation \ref{eq:matthews} fails in high dimensions, consider
what it says about a convex polytope of diameter $d$ which is roughly
spherical. Such a polytope generally cannot\emph{ }fit inside a cube
of diameter $d$. Indeed, to enclose an $n$-dimensional ball of diameter
$d$, one needs a cube with diameter $d\sqrt{n}$. Thus, as $n\rightarrow\infty$,
Equation \ref{eq:matthews} becomes quite weak for certain kinds of
diameter-$d$ sets. Equation \ref{eq:bebendorf} does not suffer from
this problem. 
\item Equation \ref{eq:bebendorf} has exponential decay instead of polynomial
decay. 
\item Equation \ref{eq:bebendorf} does not require $\Omega$ to be a convex
polytope. 
\end{enumerate}
However, Bebendorf's result also has some problems: it is not directly
applicable when the initial distribution on $X_{0}$ isn't absolutely
continuous with respect to Lebesgue measure. For example, let $\delta_{x}$
denote the degenerate initial distribution defined by $\delta_{x}(A)=\mathbb{I}_{x\in A}$.
If we try apply Equation \ref{eq:bebendorf} to the limiting case
as $g(x)dx\rightarrow\delta_{x}(dx)$, the right hand side of the
bound must become infinite. This can be somewhat remedied by using
some other technique to try to understand the short-term behavior
of the diffusion, and then applying Bebendorf's result for asymptotic
results. For example, one can pick a small $t_{0}>0$ and estimate
the distribution of $X_{0}|X_{-t_{0}}=x$, $g_{0}$. In most cases
this distribution will have a continuous density, and so one can apply
Bebendorf's inequality to the modified problem in which $g=g_{0}$
and time has been shifted by $t_{0}$. If the density of $X_{0}|X_{-t_{0}}=x$
can be accurately computed for modestly large $t_{0}$, this will
be effective \textendash{} however, in that case these kinds of estimates
are not necessary in the first place. For small values of $t_{0}$,
this method yields a very weak bound (in particular, in the limit
as $t_{0}\rightarrow0$, it fails completely). 

More generally, a rich understanding of the rate of ergodicity remains
elusive, even in this simple convex case. How does it depend upon
the dimension? How does it depend upon the initial condition? Here
we will show results of \emph{uniform }ergodicity, similar to those
of Bebendorf insofar as they do not depend upon dimension, but different
in that they apply uniformly regardless of initial condition.

It turns out that a simple one-dimensional diffusion can shed some
light on these questions. Let $\left\{ W_{t}\right\} _{t\geq0}$ denote
a one-dimensional Brownian motion. Let $\tilde{\tau}_{d}=\inf\left\{ t:\ W_{t}\notin(-d,d)\right\} $.
Note that the distribution of this object is straightforward to calculate
and analyze. For example, in \cite{hu2012hitting} it is shown that
the survival function of $\tilde{\tau}_{d}$ is given by 
\[
\mathbb{P}(t\leq\tilde{\tau}_{d}|W_{t}=k)=\FF tk\triangleq\sum_{n=0}^{\infty}e^{-\frac{\pi^{2}}{8d^{2}}\left(2n+1\right)^{2}t}\frac{4\left(-1\right)^{n}}{\pi\left(2n+1\right)}\cos\left(\frac{2n+1}{2}\times\frac{\pi k}{d}\right)
\]
Numerical estimation of this sum is straightforward and effective
in practice. Indeed, $\FF tk$ is simply the solution to the heat
equation on $[-d,d]$ with homogeneous Dirichlet boundary conditions
and initial condition $\FF 0k=\mathbb{I}_{k\in(-d,d)}$; this partial
differential equation is well-understood (cf. \cite{cannon1984one}).
It is also easy to bound $F_{d}$ using the moment generating function
of $\tilde{\tau}$ and a Chernoff bound; the moment generating function
can be deduced by an application of the Kac moment formula, yielding
$\mathbb{E}\left[e^{\gamma\tilde{\tau}_{d}}|W_{t}=k\right]=\cos(\sqrt{2\gamma}k)/\cos(\sqrt{2\gamma}d)$
for any $\gamma\leq\pi^{2}/8d^{2}$ (cf. \cite{fitzsimmons1999kac,khas1959positive}).

We can use the distribution of $\tilde{\tau}_{d}$ to help us understand
the rate of ergodicity for convex domains: 
\begin{thm*}
Let $\Omega\subset\mathbb{R}^{n}$ bounded, open, convex, with diameter
$d$. Let $\left\{ \pt tx{dy}\right\} _{t\geq0,x\in\Omega}$ denote
the transition measures of Brownian motion trapped inside $\Omega$
by normally-reflecting barriers. Then 
\begin{align*}
\tv{\pt tx{\cdot}-\pt ty{\cdot}} & \leq\FF{4t}{d-\left|x-y\right|}\\
\tv{\pt tx{\cdot}-\sigma} & \leq\int\FF{4t}{d-\left|x-y\right|}\sigma(dy)\leq\FF{4t}0
\end{align*}
This very last bound is tight within a factor of 2. In, particular,
taking the special case that $\Omega=\left[0,d\right]\subset\mathbb{R}$,
we have that $\FF{4t}0\leq2\tv{\pt t0{\cdot}-\sigma}$. 
\end{thm*}
\begin{proof}
We will defer the proof to Section 3.
\end{proof}
Notice that the leading $\exp\left(-\pi^{2}t/2d^{2}\right)$ rate
in $\FF{4t}{\cdot}$ is the same as the rate given by Bebendorf in
Equation \ref{eq:bebendorf}. This is no accident. Both quantities
reflect the spectral gap for the Neumann Laplacian on the interval
$[0,d]$, namely $\pi/d$.

The author's particular interest in this problem arose from a question
about hitting probabilities. Let $A,B$ denote two open disjoint subsets
of $\Omega$. Let $T=\inf\left\{ t:\ X_{t}\in A\cup B\right\} $,
and consider the problem of estimating $u(x)=\mathbb{P}(X_{T}\in\partial A|X_{0}=x)$.
In general it can be quite tricky to analyze $u$. However, there
are some circumstances in which it simplifies considerably. Let $x\in\Omega$
denote some point such that $\mathbb{P}(T>t|X_{0}=x)\approx1$ and
$\tv{\pt tx{\cdot}-\sigma}\approx0$. Then it is easy to show that
$u(x)\approx\int u(y)\sigma(dy)$. This can greatly simplify both
the numerical and the analytic investigation of $u$. It was for this
reason that we wanted to closely investigate the rate of ergodicity.

The remainder of this article is divided into three sections: 
\begin{enumerate}
\item \setcounter{enumi}{1}\emph{Known results. }We give a rigorous definition
for reflecting Brownian motion in a convex set and formalize some
aspects of our introductory exposition. We summarize known results,
look at the equations governing $\pt tx{dy}$, and see how the work
by Bebendorf yields the rate of convergence found in Equation \ref{eq:bebendorf}.
Finally, we will examine a coupling idea whose first rigorous construction
is due to Atar and Burdzy (cf. \cite{atar2004neumann}). 
\item \emph{Application of known results. }Here we prove our main theorem,
using the coupling construction of Atar and Burdzy. 
\item \emph{Conclusions. }We will consider possible directions for future
research. 
\end{enumerate}

\section{Known results}

The theory of reflected Brownian motion in convex sets is fairly well-developed.
To get a sense of this history, we will here recall Tanaka's early
work on the subject: 
\begin{defn}
\label{def:RBM}Let $\Omega\subset\mathbb{R}^{n}$ denote an open
bounded set. Let $\left\{ W_{t}\right\} _{t\geq0}$ denote an $n$-dimensional
Brownian motion adapted to $\left\{ \mathcal{F}_{t}\right\} _{t\geq0}$.
Let us say there exists 
\end{defn}

\begin{itemize}
\item a continuous process $\left\{ X_{t}\right\} _{t\geq0}\subset\bar{\Omega}$
which is adapted to $\left\{ \mathcal{F}_{t}\right\} _{t\geq0}$ 
\item a positive locally finite random measure $\mu$ on $[0,\infty)$ 
\item a random function $\boldsymbol{n}:\ \mathbb{R}^{+}\rightarrow\mathbb{R}^{n}$ 
\end{itemize}
such that

\begin{gather*}
X_{t}=W_{t}+\int_{0}^{t}\boldsymbol{n}(s)\mu(ds)\\
\mu\left(\left\{ t:\ X_{t}\notin\partial\Omega\right\} \right)=0
\end{gather*}
for every $t$ and\textit{ }\textbf{\textit{$\boldsymbol{n}(t)$ }}is
a normal vector of a supporting hyperplane\footnote{That is, $\left|\boldsymbol{n}\right|=1$ and $\Omega\subset\left\{ y:\ \left\langle \boldsymbol{n},y-X_{t}\right\rangle \geq0\right\} $}\textit{
of $\Omega$} at the point $X_{t}$ for $\mu$-almost-every value
of $t$. Then we will call $X$ a \textbf{\textit{\emph{reflecting
Brownian motion in $\Omega$ driven by $W$.}}}\emph{ }
\begin{thm}
\label{thm:(Tanaka's-1979).-Let}Let $W$ denote any $n$-dimensional
Brownian motion with $W_{0}\in\Omega$. If $\Omega\subset\mathbb{R}^{n}$
is convex then there is a pathwise-unique reflecting Brownian motion
in $\Omega$ driven by $W$. 
\end{thm}

\begin{proof}
Cf. \cite{tanaka1979stochastic}.
\end{proof}
Here we summarize some well-known facts about the process $X$. 
\begin{lem}
{[}Properties of Reflected Brownian motion in a convex set $\Omega${]}
\label{lem:(Properties-of-Reflected} Let $W$ denote an $\left\{ \mathcal{F}_{t}\right\} _{t\geq0}$-adapted
Brownian motion and let $X$ denote a reflecting Brownian motion in
$\Omega$ driven by $W$. If $\Omega$ is convex, then 
\begin{enumerate}
\item $\left\{ X_{t}\right\} _{t\geq0}$ is a strongly $\left\{ \mathcal{F}_{t}\right\} _{t\geq0}$-adapted
homogeneous Markov process. Let $\left\{ \pt tx{dy}\right\} _{t\geq0,x\in\bar{\Omega}}$
denote the transition measures of this process. The process $X$ is
reversible with respect to $\sigma$, i.e. $\int_{x\in A}\int_{y\in B}\sigma(dx)\pt tx{dy}=\int_{x\in B}\int_{y\in A}\sigma(dx)\pt tx{dy}$.
In particular, $\sigma$ is a stationary distribution of $X$. 
\item $X$ is uniformly ergodic with stationary distribution $\sigma$,
i.e. 
\[
\lim_{t\rightarrow\infty}\sup_{x}\tv{\pt tx{\cdot}-\sigma}=0
\]
\item Let $\lambda\geq0$ denote any constant such that 
\[
\int fdx=0\implies\lambda\ltn f\leq\sqrt{\int\left|\nabla f(x)\right|^{2}dx}
\]
for weakly differentiable functions $f:\ \Omega\rightarrow\mathbb{R}$.
Then for any density $g\in\mathscr{L}^{2}$ (i.e. $g\geq0$ and $\int g(x)dx=1$),
we have that 
\[
\left|\mathbb{E}\left[f(X_{t})|X_{0}\sim\rho\right]-\mathbb{E}\left[f(Y)\right]\right|\leq\exp\left(-\frac{1}{2}\lambda^{2}t\right)\ltn{g-\frac{1}{\mathrm{Vol}(\Omega)}}\ltn f
\]
where $Y\sim\sigma$ and $\rho(dx)=g(x)dx$. 
\end{enumerate}
\end{lem}

\begin{proof}
These results are well-known. We relate them at a high level here
for the convenience of the reader.

The key is to grasp the connection between $X$ and a certain so-called
``Dirichlet form,'' 
\[
\mathscr{E}(f,g)=\frac{1}{2}\int_{\Omega}\left\langle \nabla f,\nabla g\right\rangle dx
\]
which is understood as a bilinear form on the Sobolev space $H^{1}$
of weakly differentiable functions on $\Omega$. The arc of this connection
is the content of the treatise \cite{fukushima2010dirichlet}. We
will only sketch it briefly in this paragraph. Let $\mathscr{L}^{2}(\Omega)$
denote the space of square-integrable measurable functions on $\Omega$,
equipped with the inner product $\left\langle f,g\right\rangle _{\mathscr{L}^{2}}=\int f(x)g(x)dx$.
One can find a unique non-negative definite operator $A:\ H^{1}\rightarrow\mathscr{L}^{2}$
such that $\mathscr{E}(f,g)=\int\left(Af\right)\left(Ag\right)dx$.
It turns out that $A$ is self-adjoint. One one can thus obtain a
family of operators of the form $T_{t}:\ \mathscr{L}^{2}\rightarrow\mathscr{L}^{2}$,
uniquely defined as 
\[
T_{t}=e^{-A^{2}t}
\]
Note that $e^{-A^{2}t}$ is defined on all of $\mathscr{L}^{2}$ even
though $A$ is only defined on $H^{1}$; we refer the reader to \cite{schmudgen2012unbounded}
for a very clear introduction to these considerations. If $\Omega$
has lipschitz boundary (i.e. the boundary can locally be represented
as the epigraph of a lipschitz function), one can then (not-necessarily-uniquely)
define a strong Markov process $\left\{ Y_{t}\right\} _{t\geq0}$
such that 
\[
\mathbb{E}\left[f(Y_{t+s})|Y_{t}=y\right]-\left(T_{t}f\right)(y)
\]
almost surely with respect to Lebesgue measure, for every $f\in\mathscr{L}^{2}$.
In this case we may say that $Y$ is ``weakly determined'' by $\mathscr{E}$.

It is shown in \cite{bass1990semimartingale} that if $\Omega$ is
bounded with lipschitz boundary, then \emph{any} process which is
weakly determined by $\mathscr{E}$ will be a reflecting Brownian
motion driven by some Brownian motion $W$. It is also shown that
at least one such process exists. Since a bounded convex set automatically
has a lipschitz boundary (cf. Corollary 1.2.2.3 of \cite{grisvard2011elliptic})
and Tanaka showed that reflecting Brownian motion on a convex set
is uniquely defined, it follows that the process $X$ must be weakly
determined by $\mathscr{E}$.

This allows us to prove all three of our claims: 

Since \cite{fukushima2010dirichlet} shows that every process with
lipschitz boundary that is weakly determined by $\mathscr{E}$ is
a strong Markov process, it follows that $X$ is a strong Markov process.
The reversibility of $X$ with respect to $\sigma$ then follows from
the fact that $T_{t}$ is self-adjoint as an operator on $\mathscr{L}^{2}(\Omega,\sigma)$
(this, in turns follows from the fact that $A$ is self-adjoint). 

In \cite{burdzy2006traps} it is shown that if a process $\left\{ Y_{t}\right\} _{t\geq0}$
is weakly determined by $\mathscr{E}$ and $\Omega$ is convex, then
\[
\lim_{t\rightarrow\infty}\sup_{y}\sup_{A}\left|\mathbb{P}(Y_{t}\in A|Y_{0}=y)-\sigma(A)\right|=0
\]
Thus the same follows for our process, $X$. 

Let $\lambda$ denote any constant so that $\int fdx=0\implies\lambda^{2}\ltn f^{2}\leq\mathscr{E}(f,f)$.
Recall that we have said there is a unique operator $A:\ H^{1}\rightarrow\mathscr{L}^{2}$
such that $\mathscr{E}(f,g)=\int\left(Af\right)\left(Ag\right)dx$.
In particular, $\mathscr{E}(f,f)=\ltn{Af}^{2}$. We may thus rephrase
our understanding of $\lambda$ by saying that $\int fdx=0\implies\lambda\ltn f\leq\ltn{Af}$.
Using spectral methods it is thus straightforward to show that 
\begin{equation}
\int fdx=0\implies\ltn{e^{-A^{2}t}f}^{2}\leq e^{-\lambda^{2}t}\ltn f^{2}\label{eq:spectralboundofpoincare}
\end{equation}
Using this and Cauchy-Schwarz, one can readily show that 
\[
\left|\mathbb{E}\left[f(X_{t})|X_{0}\sim\rho\right]-\mathbb{E}\left[f(Y)\right]\right|\leq e^{-\frac{1}{2}\lambda^{2}}\ltn{g-\frac{1}{\mathrm{Vol}(\Omega)}}\ltn f
\]
\end{proof}
\begin{cor}
Let $Y\sim\sigma$, $\rho(dx)=g(x)dx$, $g\geq0$ and $\rho(\Omega)=1$.
Then 
\[
\left|\mathbb{E}\left[f(X_{t})|X_{0}\sim\rho\right]-\mathbb{E}\left[f(Y)\right]\right|\leq\exp\left(-\frac{1}{2}\left(\frac{\pi}{d}\right)^{2}t\right)\ltn{g-\sigma}\ltn f
\]
\end{cor}

\begin{proof}
The work in \cite{bebendorf2003note} shows that $\lambda=\pi/d$
fills the required role for statement 3 of Lemma \ref{lem:(Properties-of-Reflected}. 
\end{proof}
This last corollary gives a satisfying grip on the $\mathscr{L}^{2}$
ergodic convergence for Brownian motion in convex domains. Our endeavor
here is to complement this with a comparable analysis of the total
variation convergence.

Towards this end, we will be employ a coupling construction. That
is, we will construct a joint process $\left\{ X_{t},Y_{t}\right\} _{t\geq0}$
so that $X$ and $Y$ both carry the law of reflecting brownian motion,
but each has a different initial condition. We will construct them
in such a way that $\tau=\inf\left\{ t:\ X_{t}=Y_{t}\right\} $ is
almost surely finite:
\begin{thm}
\label{thm:atar}Fix any bounded convex set $\Omega$. Let $\left\{ W_{t}\right\} _{t\geq0}$
denote a brownian motion. Then there exists a solution to the equations
\begin{equation}
\begin{array}{ll}
X_{t} & =x+W_{t}+L_{t}\in\bar{\Omega}\\
Y_{t} & =y+Z_{t}+M_{t}\in\bar{\Omega}\\
Z_{t} & =W_{t}-\int_{0}^{t}2\eta_{s}\left\langle \eta_{s},dW_{s}\right\rangle \\
\eta_{t} & =\frac{X_{t}-Y_{t}}{\left|X_{t}-Y_{t}\right|}
\end{array}\qquad\qquad\qquad\begin{array}{ll}
\left|L\right|_{t} & =\int_{0}^{t}\mathbb{I}_{X_{s}\in\partial\Omega}d\left|L\right|_{s}<\infty\\
L_{t} & =-\int_{0}^{t}\boldsymbol{n}_{L}(s)d\left|L\right|_{s}\\
\left|M\right|_{t} & =\int_{0}^{t}\mathbb{I}_{Y_{s}\in\partial\Omega}d\left|M\right|_{s}<\infty\\
M_{t} & =-\int_{0}^{t}\boldsymbol{n}_{M}(s)d\left|M\right|_{s}
\end{array}\label{eq:atar}
\end{equation}
defined in a pathwise unique way up until the time $\tau=\inf\left\{ t:\ X_{t}=Y_{t}\right\} $.
Here $d\left|L\right|_{s}$ plays the role of the measure $\mu$ in
Definition \ref{def:RBM}; we require that \textbf{$\boldsymbol{n}_{L}(s)$
}is a normal vector of supporting hyperplanes of $\Omega$ at $X_{s}$,
$d\left|L\right|_{s}$-almost surely. Likewise for $\boldsymbol{n}_{M},Y_{s},d\left|M\right|_{s}$.
In particular, let
\[
\tilde{Y}_{t}=\begin{cases}
Y_{t} & t\leq\tau\\
X_{t} & t\geq\tau
\end{cases}
\]
Then $\left\{ \left(X_{t},\tilde{Y}_{t}\right)\right\} _{t\geq0}$
constitute a strongly Markovian process, and both $X$ and $Y$ are
reflecting Brownian motions.
\end{thm}

\begin{proof}
We refer the reader to the work of Atar and Burdzy in \cite{atar2004neumann}.
Note that although this article focuses on the case that $\partial\Omega$
is smooth, it also mentions that all of the reasoning goes through
for any set which is ``admissible'' according to the lights of work
by Lions and Sznitman, \cite{lions1984stochastic}. Convex sets are
indeed ``admissible'' according to Remark 3.1 of the work by Lions
and Sznitman. 

We emphasize that even though $X$ and $Y$ in this theorem are profoundly
coupled, individually they both behave like Brownian motions trapped
inside $\Omega$ by reflecting boundaries. It is also worth emphasizing
that there are two completely different conceptual ``reflections''
at play here: 
\end{proof}
\begin{enumerate}
\item The normally reflecting boundaries keep $X,Y$ inside $\Omega$ 
\item The mirror coupling causes $Y$ to generally behave like the mirror
image of $X$, reflected over a plane halfway between $X$ and $Y$. 
\end{enumerate}
These two kinds of reflections may interact when $X$ or $Y$ hits
a point in $\partial\Omega$. In this case the direction of reflection
(which is mathematically expressed as $\eta_{t}$) may rotate.

\section{Application of known results}

To show our main theorem, we must understand the distribution of the
coupling time in the mirror construction from Theorem \ref{thm:atar}.
Exact calculation of this distribution may be impossible, but some
bounds are straightforward to obtain: 
\begin{lem}
\label{lem:timing}Let $\Omega,X,Y,\tilde{Y},\tau$ be as in Theorem
\ref{thm:atar}. Let $d$ denote the diameter of $\Omega$. Then 
\begin{align*}
\mathbb{P}_{x}\left(t\leq\tau\right) & \leq\FF{4t}{d-\left|x-y\right|}
\end{align*}
\end{lem}

\begin{proof}
Let us consider 
\[
R_{t}=\left|X_{t\wedge\tau}-Y_{t\wedge\tau}\right|
\]
We now apply the fact that Equation \ref{eq:atar} holds for for $t\leq\tau$.
This yields that 
\[
R_{t}=\left|x-y\right|+\int_{0}^{t\wedge\tau}\left\langle \eta_{s},\boldsymbol{n}_{M}(s)d\left|M\right|_{s}-\boldsymbol{n}_{L}(s)d\left|L\right|_{s}\right\rangle +2\left\langle \eta_{s},dW_{s}\right\rangle 
\]
Dambis-Dubins-Schwarz then yields that we can find some one-dimensional
Brownian motion $B$ such that 
\[
R_{t}=\left|x-y\right|+\int_{0}^{t\wedge\tau}\left\langle \eta_{s},\boldsymbol{n}_{M}(s)d\left|M\right|_{s}-\boldsymbol{n}_{L}(s)d\left|L\right|_{s}\right\rangle +B_{4\left(t\wedge\tau\right)}
\]
Let us now inspect 
\[
\Phi_{t}=\int_{0}^{t\wedge\tau}\left\langle \eta_{s},\boldsymbol{n}_{M}(s)d\left|M\right|_{s}-\boldsymbol{n}_{L}(s)d\left|L\right|_{s}\right\rangle 
\]
Let us focus on three properties of this object: 
\begin{enumerate}
\item $t\mapsto\Phi_{t}$ is monotone non-increasing. Indeed, recall that
$\boldsymbol{n}_{M}(s)$ is a supporting hyperplane of $\Omega$ at
$X_{s}$, $d\left|M\right|_{s}$-almost-surely. That is, 
\[
\Omega\subset\left\{ y:\ \left\langle \boldsymbol{n}_{M}(s),y-X_{t}\right\rangle \geq0\right\} 
\]
So certainly 
\[
\left\langle \eta_{s},\boldsymbol{n}_{M}(s)\right\rangle =\left\langle \frac{X_{s}-Y_{s}}{\left|X_{s}-Y_{s}\right|},\boldsymbol{n}_{M}(s)\right\rangle \leq0
\]
The same arguments apply to $-\left\langle \eta_{s},\boldsymbol{n}_{L}(s)\right\rangle $. 
\item Since the diameter of $\Omega$ is $d$, we have that $R_{t}\leq d$
for all $t$. Put another way, $\Phi_{t}\leq d-B_{4t}-\left|x-y\right|$. 
\item $\Phi_{0}=0$. 
\end{enumerate}
Putting these facts together, we obtain the overall bound of 
\[
\Phi_{t}\leq\inf_{s\leq t\wedge\tau}\left(d-B_{4s}-\left|x-y\right|\right)\wedge0
\]
And thus 
\[
R_{t}\leq\left|x-y\right|+B_{4\left(t\wedge\tau\right)}+\inf_{s\leq t\wedge\tau}\left(d-B_{4s}-\left|x-y\right|\right)\wedge0\triangleq\tilde{R}_{4t}
\]
This is useful because the law of $\tilde{R}_{t}$ is well-understood.
It is that of a one-dimensional Brownian motion with reflection at
the point $d$, initialized at $\tilde{R}_{0}=\left|x-y\right|$ (cf.
the introduction of \cite{dupuis1991lipschitz} for a useful exposition
on this point), and stopped at the time when $\tilde{R}_{t}=0$. In
particular, letting $\tilde{\tau}=\inf\left\{ t:\ \tilde{R}_{t}=0\right\} $,
we can apply the results of \cite{hu2012hitting} to argue that 
\begin{align*}
\mathbb{P}_{x}\left(t\leq\tilde{\tau}\right) & =\FF t{d-\left|x-y\right|}
\end{align*}
Moreover, since $R_{t}\leq\tilde{R}_{4t}$ for every $t$, it follows
that $\tilde{\tau}\geq4\tau$. Our result follows immediately.
\end{proof}
This leads to our main theorem, which we restate here for the convenience
of the reader: 
\begin{thm*}
Let $\Omega\subset\mathbb{R}^{n}$ bounded, open, convex, with diameter
$d$. Let $\left\{ \pt tx{dy}\right\} _{t\geq0,x\in\Omega}$ denote
the transition measures of Brownian motion trapped inside $\Omega$
by normally-reflecting barriers. Then 
\begin{align*}
\tv{\pt tx{\cdot}-\pt ty{\cdot}} & \leq\FF{4t}{d-\left|x-y\right|}\\
\tv{\pt tx{\cdot}-\sigma} & \leq\int\FF{4t}{d-\left|x-y\right|}\sigma(dy)\leq\FF{4t}0
\end{align*}
This very last bound is tight within a factor of 2. In, particular,
taking the special case that $\Omega=\left[0,d\right]\subset\mathbb{R}$,
we have that $\FF{4t}0\leq2\tv{\pt t0{\cdot}-\sigma}$.
\end{thm*}
\begin{proof}
The total variation distance between $\pt tx{\cdot}-\pt ty{\cdot}$
is easy to bound using Lemma \ref{lem:timing}. Recall that $X_{t},\tilde{Y}_{t}$
were both reflecting Brownian motions, with $X_{0}=x$, $\tilde{Y}_{0}=y$,
and $X_{t}=\tilde{Y}_{t}$ for $t\geq\tau$. The lemma then shows
that $\mathbb{P}\left(t<\tau\right)\leq\FF{4t}{d-\left|x-y\right|}$.
We thus obtain our first total variation bound: 
\begin{align*}
\tv{\pt tx{\cdot}-\pt ty{\cdot}} & =\sup_{A}\left|\mathbb{P}\left(X_{t}\in A\right)-\mathbb{P}\left(\tilde{Y}_{t}\in A\right)\right|\\
 & \leq\mathbb{P}\left(X_{t}\neq\tilde{Y}_{t}\right)\leq\mathbb{P}\left(t<\tau\right)\leq\FF{4t}{d-\left|x-y\right|}
\end{align*}
The second total variation bound then follows immediately from the
fact that $\sigma$ is the stationary distribution of the process,
i.e. $\int\pt txA\sigma(dx)=\sigma(A)$. The second total variation
bound is thus a kind of ``average'' of the first variation bound:
\begin{align*}
\tv{\pt tx{\cdot}-\sigma} & =\sup_{A}\left|\pt txA-\int\pt tyA\sigma(dy)\right|\\
 & \leq\int\tv{\pt tx{\cdot}-\pt ty{\cdot}}\sigma(dy)\\
 & \leq\int\FF t{d-\left|x-y\right|}\sigma(dy)
\end{align*}
Finally, it is well-known that $\sup_{k}\FF tk=\FF t0$, so we obtain
the overall bound of $\tv{\pt tx{\cdot}-\sigma}\leq\FF t0$.

We now turn our attention to the one-dimensional case, to see that
this last bound is tight within a factor of 2. Let us assume $\Omega=\left[0,d\right]$
and $X_{0}=0$. Let 
\[
V(t,x)=\mathbb{P}(X_{t}\leq d/2|X_{0}=x)=\int\mathbb{I}_{y\leq d/2}\pt tx{dy}
\]
It is well-known that in this simple one-dimensional case, $V$ is
the unique solution to 
\begin{align*}
\frac{\partial}{\partial t}V(t,x) & =\frac{1}{2}\frac{\partial^{2}}{\partial x^{2}}V(t,x)\\
\frac{\partial}{\partial x}V(t,0)=\frac{\partial}{\partial x}V(t,d) & =0\\
V(0,x) & =\mathbb{I}_{x\leq d/2}
\end{align*}
Note that since $V(0,x)$ is discontinuous, we need to construe these
equations weakly; we refer the reader to \cite{fukushima2010dirichlet}
for a detailed account of how this can be done and why this partial
differential equation governs the behavior of $V$. The key idea is
that $V(t,\cdot)=e^{-A^{2}t}V(0,\cdot)$ where $e^{-A^{2}t}$ is the
$\mathscr{L}^{2}$ operator whose construction we outlined in Lemma
\ref{lem:(Properties-of-Reflected}. Once the problem is properly
formulated, a fourier expansion immediately yields the solution: 
\begin{align*}
V(t,x) & =\frac{1}{2}+\frac{1}{2}\sum_{n=0}^{\infty}e^{-\frac{\pi^{2}}{2d^{2}}\left(2n+1\right)^{2}t}\frac{4\left(-1\right)^{n}}{\pi\left(2n+1\right)}\cos\left(\left(2n+1\right)\times\frac{\pi x}{d}\right)\\
 & =\frac{1}{2}+\frac{1}{2}\FF{4t}{2x}
\end{align*}
In particular, 
\begin{align*}
\tv{\pt t0{\cdot}-\sigma} & =\sup_{A}\left|\mathbb{P}(X_{t}\in A|X_{0}=0)-\sigma(A)\right|\\
 & \geq\left|\mathbb{P}(X_{t}\leq d/2|X_{0}=0)-\frac{1}{2}\right|\\
 & =\frac{1}{2}\FF{4t}0
\end{align*}
\end{proof}

\section{Conclusions}

Among all convex sets with diameter $d$, this work begins to suggest
that the one-dimensional interval (e.g. $\Omega=[0,d]$) may provide
a kind of worst-case scenario for mixing rates. This is helpful because
the one-dimensional interval is easy to analyze.

Unfortunately, ``typical'' high-dimensional sets of diameter $d$
may mix much faster than our bounds would suggest. Thus, Equation
\ref{eq:matthews} from \cite{matthews1990mixing} seemed to imply
that mixing might get slower in high dimensions, the bound in our
theorem does not depend upon the dimension, but the reality may be
that mixing typically gets faster in higher dimensions. For example,
preliminary analysis suggests that $n$-dimensional Brownian motion
in the unit $n$-dimensional ball mixes quite a bit faster than one-dimensional
Brownian motion on $[0,d]$, especially as $n\rightarrow\infty$.
On the other hand, we are simply not sure about the total-variation
mixing rate for Brownian motion in a high-dimensional cube. Are there
simple ways to improve our bounds when we know more about $\Omega$?
This is a possible direction of future research.

In another direction, it should be possible to extend this basic method
of proof to accommodate a wide variety of Ito diffusions, beyond Brownian
motion. Mirror couplings are available for many such diffusions; we
refer the reader to \cite{lindvall1986coupling} and the many papers
which have cited it. For every such mirror coupling process, $\left\{ X_{t},Y_{t}\right\} _{t\geq0}$,
one can analyze the one-dimensional process $R_{t}=\left|X_{t}-Y_{t}\right|$.
By applying Ito's lemma, Dambis-Dubins-Schwarz, and taking bounds,
one can often obtain a stochastic differential inequality of the form
$dR_{t}\leq d\tilde{R}_{t}$, where $\tilde{R}_{t}$ is some semimartingale
that is better-understood. Applying stochastic differential inequality
results such as those found in \cite{ding1998new}, one can then obtain
bounds for the high-dimensional process with some simple one-dimensional
process.

For example, consider the stochastic differential equation $dX_{t}=\mu(X_{t})dt+W_{t}$
where $\mu$ is some Lipschitz vector field. Similar to the technique
we used, let $dY_{t}=\mu(Y_{t})dt-dW_{t}-2\eta_{t}\left\langle \eta_{t},dW_{t}\right\rangle $,
where $\eta$ is a normalized version of $X-Y$. As in our case, if
we define $B_{4t}=\int_{0}^{t}2\left\langle \eta_{s},dW_{s}\right\rangle $
then we can show that $B$ is a one-dimensional Brownian motion. Finally,
let $\Gamma$ denote some Lipschitz function satisfying $\Gamma(r)\geq\sup_{\left|x-y\right|=r}\left\langle x-y,\mu(x)-\mu(y)\right\rangle $,
and consider the simple one-dimensional stochastic differential equation
\[
Z_{t}=\left|x-y\right|+\int_{0}^{t}\frac{\Gamma(Z_{s})}{Z_{s}}ds+B_{4t}
\]
Using the results from \cite{ding1998new}, one can readily show that
$R_{t}\leq Z_{t}$. Thus 
\[
\inf\left\{ t:\ X_{t}=Y_{t}\right\} \leq\inf\left\{ t:\ Z_{t}=0\right\} 
\]
In short, by analyzing the simple one-dimensional diffusion $Z$,
one can estimate coupling times for extremely complex and high-dimensional
processes. These coupling times can then be used to estimate rates
of ergodicity. Of course, these kinds of estimates may be quite poor
in some cases (e.g. consider the catastrophic case that $\sup_{\left|x-y\right|=r}\left\langle x-y,\mu(x)-\mu(y)\right\rangle =\infty$).
This difficulty is related to the problem with which we began these
conclusions: in some cases it may be very difficult to get high-quality
bounds using only a one-dimensional diffusion.

We have seen that one-dimensional diffusions can be used to analyze
very high-dimensional diffusions, although the bounds may not be ideal
in certain cases. Might it be possible to improve this basic technique
to allow two or three-dimensional diffusions to give insight into
high-dimensional processes? This is an intriguing question for future
research. 

\bibliographystyle{plain}
\bibliography{refs}

\end{document}